\documentclass[a4paper,12pt]{article}
\usepackage{amsthm}
\usepackage{amsmath}
\usepackage{amssymb}
\usepackage{amsfonts}
\usepackage{latexsym}
\usepackage{mathrsfs}
\usepackage{mathtools}
\usepackage{xcolor}
\usepackage[T1]{fontenc}
\usepackage{lmodern}
\usepackage{bm}
\usepackage{enumitem}
\usepackage{titlesec}

\usepackage[dvipdfmx]{hyperref}
\hypersetup{
setpagesize=false,
 bookmarksnumbered=true,
 bookmarksopen=true,
 colorlinks=false,
}

\renewcommand{\Re}{\operatorname{Re}}
\renewcommand{\Im}{\operatorname{Im}}
\newcommand{\ci}{\mathrm{i}}
\newcommand{\ce}{\mathrm{e}}
\newcommand{\cd}{\mathrm{d}}

\DeclarePairedDelimiter{\abs}{\lvert}{\rvert} 
\DeclarePairedDelimiter{\rbra}{(}{)} 
\DeclarePairedDelimiter{\cbra}{\{}{\}} 
\DeclarePairedDelimiter{\sbra}{[}{]} 

\usepackage{enumitem}

\numberwithin{equation}{section}
\allowdisplaybreaks

\theoremstyle{plain}
\newtheorem{Thm}{Theorem}[section]
\newtheorem{Lem}[Thm]{Lemma}
\newtheorem{Prop}[Thm]{Proposition}

\theoremstyle{definition}

\theoremstyle{remark}
\newtheorem{Eg}[Thm]{Example}

\theoremstyle{plain}
\newtheorem*{Thm*}{Theorem}
\newtheorem*{Lem*}{Lemma}
\newtheorem*{Prop*}{Proposition}
\newtheorem*{Cor*}{Corollary}
\newtheorem*{Conj*}{Conjecture}
\theoremstyle{definition}
\newtheorem*{Ass*}{Assumption}
\newtheorem*{Def*}{Definition}
\newtheorem*{Rem*}{Remark}
\theoremstyle{remark}
\newtheorem*{Eg*}{Example}

\titleformat*{\section}{\large\bfseries}
\usepackage[margin=15truemm]{geometry}
\title{ The asymptotic behavior of the renormalized zero resolvent 
		of L\'{e}vy processes under regular variation conditions }
\author{ Kouji Yano and Mingdong Zhao }
\date{}
\begin{document}
\maketitle

\begin{abstract}
As an analogue to the explicit formula in the stable case, the asymptotic behavior at the origin of the renormalized zero resolvent
of one-dimensional L\'{e}vy processes is studied under certain regular variation conditions 
on the L\'{e}vy-Khinchin exponent and the L\'{e}vy measure.
\end{abstract}

\section{Introduction} \label{sec1}
Let $ \rbra*{X , \mathbb{P}} $ denote the canonical representation of a one-dimensional L\'{e}vy process starting at $0$. 
The L\'{e}vy process may be characterized by the L\'{e}vy-Khinchin exponent $\varPsi$:
\begin{align}
	\varPsi\rbra*{\lambda} \coloneqq - \log\rbra*{ \mathbb{E} \sbra*{ \ce^{\ci \lambda X_1} } }
	= a \lambda^2 + \ci b \lambda + \int_{\mathbb{R}} \rbra*{ 1 - \ce^{\ci \lambda x} + \ci \lambda x \cdot1_{ \{\abs*{x} < 1\} } } \nu \rbra*{ \cd x }, \quad
	\lambda \in \mathbb{R}, \notag 
\end{align}
where $a\ge0$ is the Gaussian coefficient, $b\in\mathbb{R}$ is a constant and $\nu$ is the L\'{e}vy measure (i.e. a 
measure on $\mathbb{R}$ satisfying $\int_{\mathbb{R}}\rbra*{x^2\land 1}\nu\rbra*{\cd x}<\infty$ and $\nu\rbra*{\{0\}}=0$). 
We define the exponents $\theta$ and $\omega$ as follows:
\begin{align}
\theta\rbra*{\lambda} &\coloneqq \Re\varPsi\rbra*{\lambda} = a\lambda^2 + \int_\mathbb{R} \rbra*{1 - \cos \lambda x} \nu\rbra*{\cd x} \quad(\lambda\in\mathbb{R}), \label{eq.theta} \\
\omega\rbra*{\lambda} &\coloneqq \Im\varPsi\rbra*{\lambda} = b\lambda + \int_\mathbb{R} \rbra*{\lambda x \cdot1_{\{\abs*{x} < 1\}} - \sin \lambda x} \nu\rbra*{\cd x} \quad(\lambda\in\mathbb{R}). \label{eq.omega} 
\end{align}
We consider the following Assumption~\ref{ass.A}:
\begin{enumerate}[label=\textbf{(A)}]
	\item \label{ass.A}
		For every $ q > 0 $, it holds that
		\begin{align}
			\int_0^{\infty} \frac{ 1 }{\abs*{ q + \varPsi \rbra*{ \lambda } } }  \, \cd \lambda < \infty \notag.
		\end{align}  
\end{enumerate}
Under Assumption~\ref{ass.A}, the $q$-resolvent density (denoted by $r_q$) exists for each $q>0$; see Proposition~\ref{prop2.1}. 
The renormalized zero resolvent \textit{h} is defined by 
\begin{align} 
h\rbra*{x} \coloneqq \lim_{q\rightarrow0+} \rbra*{r_q\rbra*{0} - r_q\rbra*{-x}} ,
\end{align} 
where the limit exists and is finite under Assumption~\ref{ass.A}; see Section~\ref{sec2}.
Moreover, it is known that $h$ has an integral representation under an assumption which is stronger than Assumption~\ref{ass.A}:
\begin{Thm} [Tsukada~\cite{Tsukada2018}] \label{thm1.1}
	Suppose that Assumption~\ref{ass.T} holds:
	\begin{enumerate}[label=\textbf{(T)}]
			\item \label{ass.T}
				Assumption~\ref{ass.A} holds as well as 
					\begin{align}
						\int_0^1 \abs*{\Im\rbra*{\frac{\lambda}{\varPsi\rbra*{\lambda}}}} \cd \lambda<\infty.\label{int.ass.T} 
					\end{align}
		\end{enumerate}
	Then the renormalized zero resolvent $h$ admits the following integral representation:
	\begin{align}
		h\rbra*{x}=\frac{1}{\pi} \int_0^\infty 
		\Re \rbra*{\frac{1-\ce^{\ci\lambda x}}{\varPsi\rbra*{\lambda}}} \cd \lambda ,\quad x\in\mathbb{R} .\label{eq.h}
		\end{align}
\end{Thm} 
  
A typical example is the class of strictly stable L\'{e}vy process of index $\alpha\in\rbra*{1,2}$. 
The L\'{e}vy measure is given as
\begin{align}
	\nu(\cd x) = 
		\begin{cases} 
			K_+ \abs*{x}^{-\alpha-1} \cd x \quad \text{on $(0,\infty)$,} \\
			K_- \abs*{x}^{-\alpha-1} \cd x \quad \text{on $(-\infty,0)$, }
		\end{cases}
\end{align}
for some constants $ K_+, K_-\ge0$ with $ K_+ +K_- >0 $, and the exponents $\theta$ and $\omega$ are given by
\begin{align}
	\theta\rbra*{\lambda} =c_\theta\lambda^\alpha\quad\text{and}\quad\omega\rbra*{\lambda}=c_\omega\lambda^\alpha \quad
	(\lambda>0),
\end{align}
where the constants $c_\theta$ and $c_\omega$ are given by
\begin{align}
	c_\theta = \pi\rbra*{K_+ + K_-}C_{\alpha+1}, \quad c_\omega = -c_\theta \beta \tan\frac{\pi\alpha}{2}\quad 
	\text{with}\quad \beta = \frac{K_+ - K_-}{K_+ + K_-}.
\end{align}
Here the constant $C_\alpha$ is given (see, e.g.,~\cite[Proposition 7.1]{Y3}) as 
\begin{align}
	C_\alpha \coloneqq \frac{1}{\pi} \int_0^\infty \frac{1-\cos x}{x^\alpha} \cd x =
	\frac{1}{2\Gamma\rbra*{\alpha}\sin\frac{\pi \rbra*{\alpha - 1}}{2}}\quad \rbra{\alpha\in\rbra*{1,3}}, \label{Calpha}
\end{align}
where $\Gamma(\alpha)=\int_0^\infty x^{\alpha-1}\ce^{-x} \cd x\,(\alpha>0)$ denotes the Gamma function. 
We assume $1<\alpha<2$ so that Assumption~\ref{ass.T} holds. 
By Eq.$(5.2)$ of~\cite{Y2}, the corresponding renormalized zero resolvent $h$ is given by 
\begin{align}
	h(x) = \frac{1-\beta \text{sgn}(x)}{c_\theta\rbra*{1+\beta^2 \rbra*{\tan \frac{\pi \alpha}{2}}^2} 
	\cdot\rbra*{2 \Gamma(\alpha)\cdot \rbra*{-\cos\frac{\pi\alpha}{2}}}}\abs{x}^{\alpha-1}. 
\end{align}
This function can be expressed as
\begin{align}
	h(x) = 
		\begin{cases}
			c_+\abs*{x}^{\alpha-1}\quad (x\ge0) ,\\
			c_-\abs*{x}^{\alpha-1}\quad (x<0),
		\end{cases}
		\label{stable.h}
\end{align}
where $c_\pm$ are given by
\begin{align}
	c_\pm&= \frac{1\mp \beta}{{c_\theta}\rbra*{1+\beta^2 \rbra*{\tan \frac{\pi \alpha}{2}}^2}\cdot 
	\rbra*{2 \Gamma(\alpha)\cdot\rbra*{-\cos\frac{\pi\alpha}{2}}}}\notag\\
	&= \frac{c_\theta\mp c_\theta\beta}{{c_\theta}^2\rbra*{1+\beta^2 \rbra*{\tan \frac{\pi \alpha}{2}}^2}\cdot \rbra*{2 \Gamma(\alpha)
	\sin\frac{\pi(\alpha-1)}{2}}} \notag\\
	&=\frac{C_\alpha}{\rbra*{{c_\theta}^2+{c_\omega}^2}}\rbra*{c_\theta\pm c_\omega\cdot \cot\frac{\pi\alpha}{2}}.\label{eq.c+-}
\end{align}

As an analogue to the explicit formula~\eqref{stable.h}, we would like to obtain general results concerning the asymptotic behavior at the origin 
of the renormalized zero resolvent. Let us state our main theorems. For a slowly varying function $L(\lambda)$ and constants 
$c\in\mathbb{R}$ and $\alpha\in\mathbb{R}$, we say that $f(\lambda) \sim c \lambda^\alpha L(\lambda)$ 
when $\frac{f(\lambda)}{\lambda^\alpha L(\lambda)}\rightarrow c$, even if $c=0$. This convention is standard; see, 
e.g.,~\cite[Section 1.7.2]{RegVar}.
\begin{Thm} \label{thm1.2}
	Suppose that Assumption~\ref{ass.T} holds. Suppose, in addition, that there 
	exist a slowly varying function $L$ at $\infty$ and constants $\alpha \in \rbra*{1,2}$,  
	$ c_\theta>0 $ and $ c_\omega \in \mathbb{R}$ such that 
	\begin{align}
		\theta\rbra*{\lambda} \sim c_\theta \lambda^\alpha L\rbra*{\lambda}, 
		\quad\,\omega\rbra*{\lambda}\sim c_\omega\lambda^\alpha L\rbra*{\lambda}\quad \text{ as } \lambda\rightarrow\infty. \label{con.Psi}
	\end{align}
	Then 
	\begin{align}
		h\rbra*{x} \sim c_\pm \abs*{ x }^{\alpha - 1} L\rbra*{ \frac{1}{\abs*{ x }} }^{-1} \quad\text{ as } 
		x\rightarrow0\pm,
	\end{align}
	where the coefficients $c_+$ and $c_-$ are given as Eq.~\eqref{eq.c+-}.
\end{Thm}
Theorem~\ref{thm1.2} will be proved in Section~\ref{sec3}. Note that the condition~\eqref{con.Psi} implies Assumption~\ref{ass.A} since
\begin{align}
	\abs*{\varPsi\rbra*{\lambda}} \sim \sqrt{{c_\theta}^2+{c_\omega}^2} \lambda^\alpha L\rbra*{\lambda} \quad 
	(\lambda\rightarrow\infty). \label{valid.ass.A}
\end{align}

The next result gives a sufficient condition for the regular variation 
conditions~\eqref{con.Psi} of the exponents $\theta$ and $\omega$.
\begin{Thm} \label{thm1.3}
	Suppose that the L\'{e}vy measure $\nu$ has a density: 
	$ \nu\rbra*{\cd x} = \xi\rbra*{x} \cd x $.
	Write $ \xi_\pm\rbra*{x} \coloneqq \xi\rbra*{\pm x} ,\, x\ge0 $. 
			Suppose, in addition, that the Gaussian coefficient $a=0$ and 
			that there exist constant $ \alpha \in \rbra*{1,2} $ and slowly varying functions $K_+$ and $K_-$ at $0$
			such that 
			\begin{align}
				k \coloneqq \lim_{x\rightarrow0+} \frac{K_-\rbra*{x}}{K_+\rbra*{x}} \in [0,\infty]
			\end{align}
			and
			\begin{align}
				\xi_\pm\rbra*{x}\sim x^{-\alpha-1}K_\pm\rbra*{x}\quad\text{ as }x\rightarrow0+.
			\end{align}
			Then
			\begin{align}
				\theta\rbra*{\lambda} \sim c_\theta \lambda^\alpha L\rbra*{\lambda},
				\quad \, \omega\rbra*{\lambda} \sim c_\omega \lambda^\alpha L\rbra*{\lambda} \quad \text{ as } 
				\lambda\rightarrow\infty, \label{eq.asym.Psi}
			\end{align}
			where the coefficients $ c_\theta$ and $c_\omega$ are given as 
			\begin{align} 
				c_\theta = \pi C_{\alpha+1}, \quad 
				c_\omega =\frac{1-k}{1+k} \cdot \frac{\pi C_\alpha}{\alpha},
			\end{align} 
			and $L$ is the slowly varying function at $\infty$ given as 
			\begin{align}
				L\rbra*{\lambda} = K_+\rbra*{\lambda^{-1}} + K_-\rbra*{\lambda^{-1}} \underset{\lambda\to\infty}{\sim} 
				\begin{cases}
					\rbra*{1+k} K_+\rbra*{\lambda^{-1}} \quad (k\in[0,\infty)),\\
					K_-\rbra*{\lambda^{-1}} \quad (k=\infty).
				\end{cases}
			\end{align}		
\end{Thm}
Theorem~\ref{thm1.3} will be proved in Section~\ref{sec4}. 
Note that Theorem~\ref{thm1.3} generalizes some results of Grzywny--Le\.{z}aj--Mi\'{s}ta~\cite[Theorem 3.4]{GLM}, where they assumed that the negative part of the L\'{e}vy measure only differs from the positive part up to the constant multiple. 
\begin{Eg}\label{eg1.4}
Suppose that the Gaussian coefficient $a=0$ and the L\'{e}vy measure density is given by
\begin{align}
	\xi_\pm\rbra*{x}=K_\pm x^{-\alpha-1}(1+x)^{\alpha-\beta}\quad(x>0), \label{ex1.4.setting}
\end{align}
where $K_\pm>0$, ${\beta}\in(0,2)$ and $1<\alpha<2$ are constants. 
Then it is easy to check that $\xi_\pm\rbra*{x}$ satisfies the assumptions 
of Theorem~\ref{thm1.3} with $k=\frac{K_-}{K_+}$, since $(1+x)^{\alpha-\beta}$ is slowly varying at $0$. 
Hence we obtain Eq.~\eqref{eq.asym.Psi} with 
\begin{align}
	\quad L\equiv K_++K_-, \quad
	c_\theta=\pi C_{\alpha+1}\quad \text{and}\quad
	c_\omega=\frac{K_+-K_-}{K_++K_-} \cdot \frac{\pi C_{\alpha}}{\alpha}.\label{ex.c_Psi}
\end{align}
We see that Assumption~\ref{ass.A} holds since
\begin{align}
	\abs*{\varPsi(\lambda)}\sim\sqrt{{c_\theta}^2+{c_\omega}^2}\cdot\rbra*{K_++K_-}\lambda^{\alpha}  \quad
	(\lambda\rightarrow\infty). \label{ex.Psi1}
\end{align} 
We also see that Assumption~\ref{ass.T} holds since, by Eq.~\eqref{ex1.4.setting}, we have 
\begin{align}
	\xi_\pm\rbra*{x} \sim K_\pm x^{-{\beta}-1}\quad\text{as $x\rightarrow\infty$}
\end{align}
and, by~\ref{thm6.1.i} of Theorem~\ref{thm6.1}, we have 
\begin{align}
	\theta(\lambda) \sim \pi C_{\beta+1}(K_++K_-)\lambda^\beta\quad\text{as $\lambda\downarrow0$} ,
\end{align} 
which implies 
\begin{align}
	\int_0^1 \abs*{\Im\rbra*{\frac{\lambda}{\varPsi(\lambda)}}} \cd \lambda
	\le \int_0^1 \frac{\lambda}{\abs*{\varPsi(\lambda)}} \cd \lambda
	\le \int_0^1 \frac{\lambda}{\theta(\lambda)} \cd \lambda<\infty.
\end{align}
Hence, the renormalized zero resolvent $h$ exists and can be expressed as Eq.~\eqref{eq.h} by Theorem~\ref{thm1.1}. 
Moreover, by Theorem~\ref{thm1.2}, it holds that
\begin{align}
	h(x)\sim c_\pm \abs{x}^{\alpha-1}\quad(x\rightarrow0\pm),
\end{align}
where $c_+$ and $c_-$ are given as
\begin{align} 
	c_\pm=\frac{\rbra*{K_++K_-}C_{\alpha}}{{c_\theta}^2+{c_\omega}^2}
	\rbra*{c_\theta\pm c_\omega\cot\frac{\pi\alpha}{2}},
\end{align} 
with $c_\theta$ and $c_\omega$ being given as~\eqref{ex.c_Psi}.
\end{Eg}

\subsection{Backgrounds of the renormalized zero resolvent}
The renormalized zero resolvent often appears in the penalization problems: 
a kind of long-time limit problems, which can be depicted as follows: 
for a given parametrized family $ \tau = \{\tau_\lambda\} $ of random times, which is called a \emph{clock}, and a given 
positive weight process $\rbra*{\Gamma_t,\,t\ge0}$, the penalization problem is to figure out the probability measure 
$\mathbb{Q}$ satisfying 
\begin{align}
	\mathbb{E}^\mathbb{Q}[F_t]=\lim_{\tau\rightarrow \infty}\frac{\mathbb{E}[F_t \cdot \Gamma_\tau]}{\mathbb{E}[\Gamma_\tau]}
	\label{eq.01}
\end{align}
for all $t\ge0$ and all bounded functional $F_t$ adapted to a given filtration. 

Roynette--Vallois--Yor (see~\cite{R-V-Y1},~\cite{R-V-Y2} and~\cite{R-V-Y3}) have studied the penalization problem for the 
one-dimensional standard Brownian motion $ \rbra*{B_t,\,t\ge0} $.
When we take the constant clock (i.e. $ \tau_\lambda = \lambda $) and $\Gamma_t = f \rbra*{L_t}$ for the local time $\rbra*{L_t,\,t\ge0}$ of the origin and a given positive 
integrable function $f$ on $[0,\infty)$, the limit measure $\mathbb{Q}$ is characterized by
\begin{align}
	\mathbb{E}^\mathbb{Q}[F_t] = \mathbb{E}\sbra*{\frac{M_t}{M_0}F_t} , \label{penalization.Q}
\end{align}
where $ \rbra*{M_t,t\ge0} $ is a martingale given by
\begin{align}
	M_t = f\rbra*{L_t} \abs*{B_t} + \int_0^\infty f\rbra*{L_t +u} \cd u ,\quad t\ge0 .\label{Penal.M_t}
\end{align}
Yano--Yano--Yor~\cite{Y4} generalized the local time penalization problem to the symmetric stable processes $\rbra*{X_t,\,t\ge0}$, 
where the formula~\eqref{Penal.M_t} is replaced by
\begin{align}
	M_t = f\rbra*{L_t} h\rbra*{X_t} + \int_0^\infty f\rbra*{L_t +u} \cd u ,\quad t\ge0 .\notag
\end{align}
Takeda--Yano~\cite{T-Y2023} generalized this result to the asymmetric L\'{e}vy processes with certain random clocks.

Let us explain the earlier studies of the existence of the renormalized zero resolvent.
For symmetric L\'{e}vy processes, Salminen-Yor~\cite{S-Y} showed it under Assumption~\ref{ass.A}; 
see also Yano--Yano--Yor~\cite{Y3} and Yano~\cite{Y1}. 
For asymmetric L\'{e}vy process, Yano~\cite{Y2} 
proved it under some technical assumptions, and Pant\'{i}~\cite{Pan} and Tsukada~\cite{Tsukada2018} gave some generalization. 
Recently, Takeda--Yano~\cite{T-Y2023} showed it under Assumption~\ref{ass.A}.

\subsection{Organization}
{}The remainder of this paper is organized as follows. 
In Section~\ref{sec2}, we recall several properties of L\'{e}vy processes and the renormalized zero resolvents.  
In Section~\ref{sec3}, we prove Theorem~\ref{thm1.2}. 
In Section~\ref{sec4}, we prove Theorem~\ref{thm1.3}.
In Section~\ref{sec5}, we study similar problems when the Gaussian coefficient is positive. 
In Section~\ref{sec6} as an appendix, we state similar results where the roles of $0$ and $\infty$ are switched.

\subsection{Acknowledgements}
The author would like to thank Kohki Iba for his encouragement and helpful comments. 
This research was supported by ISM 2025-ISMCRP-5007. 
The research of K. Yano was supported by JSPS KAKENHI grant 
no.'s JP24K06781, JP24K00526 and JP21H01002 and by JSPS Open Partnership Joint Research Projects grant no. JPJSBP120249936.

\section{Preliminaries} \label{sec2}
{}Denote by $\mathcal{L}^1(\mu)$ the class of integrable functions on $\mathbb{R}$ with respect to a measure $\mu$ 
on $\mathbb{R}$. When $\mu$ is the Lebesgue measure, we will abbreviate $\mathcal{L}^1(\mu)$ as $\mathcal{L}^1$.

The $q$-resolvent operator $U_q$ for $q>0$ is defined by 
\begin{align}
	U_qf\rbra*{x}\coloneqq\mathbb{E}\sbra*{\int_0^\infty e^{-qt}f\rbra*{X_t+x} \cd t} ,
\end{align}
for all non-negative measurable function $f$. 
The following result is well-known (see, e.g.,~\cite[Corollary 15.1]{Tsukada2018}): 
\begin{Prop}\label{prop2.1}
	Suppose that Assumption~\ref{ass.A} holds. Then, for all $q>0$, the $q$-resolvent density $r_q$, which is characterized by 
	\begin{align}
		U_qf\rbra{x}=\int_\mathbb{R}f(y)r_q(y-x)\,\cd y
	\end{align} 
	for all non-negative measurable function $f$, exists and a version of it is given as follows: 
	\begin{align}
		r_q (x) =\frac{1}{2\pi} \int_{\mathbb{R}} \frac{e^{-\ci \lambda x}}{q + \varPsi\rbra*{\lambda}} \cd \lambda
			= \frac{1}{\pi} \int_0^\infty \Re\rbra*{ \frac{e^{-\ci \lambda x}}{q + \varPsi\rbra*{\lambda}} } \cd \lambda 
			\quad(q>0,\,x\in\mathbb{R}).
	\end{align}
	Consequently, $r_q$ is continuous and vanishes as $\abs{x}\rightarrow\infty$.
\end{Prop}
The following result is taken from~\cite[Lemma 15.5]{Tsukada2018}.
\begin{Lem} \label{lem2.2}
Suppose that Assumption~\ref{ass.A} holds. Then it holds that
\begin{align}
	\int_0^\infty \frac{\lambda^2 \land 1}{\abs*{ \varPsi\rbra*{\lambda} }} \cd \lambda < \infty.
\end{align}
\end{Lem}
Suppose that Assumption~\ref{ass.A} holds. For $ q>0 $ and $ x\in\mathbb{R} $, we set
\begin{align}
	h_q\rbra*{x} \coloneqq r_q\rbra*{0} - r_q\rbra*{-x} = \frac{1}{2\pi} \int_\mathbb{R} \frac{1 - \ce^{\ci \lambda x}}{q + \varPsi\rbra*{\lambda}} \cd \lambda 
	= \frac{1}{\pi} \int_0^\infty \Re \rbra*{ \frac{1 - \ce^{\ci \lambda x}}{q + \varPsi\rbra*{\lambda}} } \cd\lambda.
\end{align}
It is known that the following holds (see, e.g.,~\cite[Corollary II.18]{Bertoin}): 
	\begin{align}
		\mathbb{E}_x \sbra*{\ce^{-qT_0}} =\frac{r_q(-x)}{r_q(0)}\quad(x\in\mathbb{R},\,q>0). \label{laplace.hit.time}
	\end{align}
Hence, we have $h_q\ge0$ for all $q>0$.
The renormalized zero resolvent $h$ now satisfies 
\begin{align}
	h(x)=\lim_{q\rightarrow0+}h_q(x)=\lim_{q\rightarrow0+}\rbra*{r_q\rbra*{0} - r_q\rbra*{-x}}\ge0
\end{align}
if the limit exists. Its existence is guaranteed by Takeda--Yano~\cite[Theorem 1.1]{T-Y2023} under Assumption~\ref{ass.A}.

\section{Asymptotic behavior of the renormalized zero resolvents} \label{sec3}
We say a function $L$ is slowly varying at $a=0$ or $\infty$, if $L:(0,\infty)\mapsto[0,\infty)$ satisfies  
$\frac{L\rbra*{kx}}{L\rbra*{x}} \xrightarrow[x\to a]{} 1$ for all $ k>0 $. 
Let us suppose that the assumptions of Theorem~\ref{thm1.2} hold in this section. The existence, the finiteness and the integral representation of $h$ follow from Theorem~\ref{thm1.1}.
Recall that $\varPsi = \theta + \ci \omega$ and 
\begin{align}
	h\rbra*{x} = \frac{1}{\pi}\int_0^\infty \rbra*{1 - \cos\lambda x}\frac{\theta\rbra*{\lambda}}{\abs*{\varPsi\rbra*{\lambda}}^2} \cd \lambda 
	- \frac{1}{\pi}\int_0^\infty \sin\lambda x\frac{\omega\rbra*{\lambda}}{\abs*{\varPsi\rbra*{\lambda}}^2} \cd \lambda.
\end{align}
\begin{proof}[Proof of Theorem~\ref{thm1.2}]
We show that the integrals for small $\lambda$ vanish in the limit. For a constant $r>0$, which may be taken large 
enough later, we see that 
\begin{align}
	&\lim_{x\rightarrow0}\frac{L\rbra*{\abs*{x}^{-1}}}{\abs*{x}^{\alpha-1}}
	\int_0^r\rbra*{1-\cos\lambda x}\frac{\theta\rbra*{\lambda}}{\abs*{\varPsi\rbra*{\lambda}}^2}\cd\lambda = 0,\\ 
	&\lim_{x\rightarrow0}\frac{L\rbra*{\abs*{x}^{-1}}}{\abs*{x}^{\alpha-1}}
	\int_0^r\sin\lambda x\frac{\omega\rbra*{\lambda}}{\abs*{\varPsi\rbra*{\lambda}}^2}\cd\lambda = 0.
\end{align}
To check these, noting that $1-\cos t\le t^2$ and $\abs*{\sin t}\le\abs*{t}$ hold 
for $t\in\mathbb{R}$, we have
\begin{align}
	0\le\frac{L\rbra*{\abs*{x}^{-1}}}{\abs*{x}^{\alpha-1}} 
	\int_0^r\rbra*{1-\cos\lambda x}\frac{\theta\rbra*{\lambda}}{\abs*{\varPsi\rbra*{\lambda}}^2}\cd\lambda
	\le \abs*{x}^{3-\alpha} L\rbra*{\abs*{x}^{-1}} 
	\int_0^r \frac{\lambda^2}{\abs*{\varPsi\rbra*{\lambda}}}\cd\lambda \rightarrow 0 
	\quad ( x\rightarrow0)
\end{align}
and
\begin{align}
	&\frac{L\rbra*{\abs*{x}^{-1}}}{\abs*{x}^{\alpha-1}} 
	\int_0^r\abs*{\sin\lambda x}\frac{\abs*{\omega\rbra*{\lambda}}}{\abs*{\varPsi\rbra*{\lambda}}^2}\cd\lambda
	\le \abs*{x}^{2-\alpha} L\rbra*{\abs*{x}^{-1}}
	\int_0^r\abs*{\Im\rbra*{\frac{\lambda}{\varPsi\rbra*{\lambda}}}}\cd\lambda \rightarrow 0 \quad( x\rightarrow0).
\end{align}

We study the integrals for large $\lambda$. Let $\delta\in\rbra*{0,\rbra*{\alpha-1}\land\rbra*{2-\alpha}}$. 
By Potter's theorem (see, e.g.,~\cite[Theorem 1.5.6]{RegVar}), there exists $r_1>0$ such that 
\begin{align}
	\frac{L\rbra*{y}}{L\rbra*{x}} \le 2 \max\cbra*{ \rbra*{\frac{y}{x}}^\delta,\rbra*{\frac{y}{x}}^{-\delta} } \quad
	(x,y\ge r_1).
\end{align}
In addition, there exists $r_2$ such that 
\begin{align}
	\frac{1}{2}c_\theta \le \frac{\theta\rbra*{\lambda}}{\lambda^\alpha L\rbra*{\lambda}} \le 2 c_\theta,\,\,
	\quad \frac{1}{2}\abs*{c_\omega} \le \frac{\abs*{\omega\rbra*{\lambda}}}{\lambda^\alpha L\rbra*{\lambda}} 
	\le 2\abs*{c_\omega} + 1\quad(\lambda\ge r_2).
\end{align}
We now take $r=\max\cbra{r_1,r_2}$ and then, for $\lambda\ge r$, it holds that
\begin{align}
\frac{\abs*{\varPsi\rbra*{\lambda}}}{\lambda^\alpha L\rbra*{\lambda}}
\ge\frac{1}{2}\sqrt{{c_\theta}^2+{c_\omega}^2} 
\end{align} 
and 
\begin{align}
	0\le\frac{\theta\rbra*{\lambda}}{\abs*{\varPsi\rbra*{\lambda}}^2} \lambda^\alpha L\rbra*{\lambda} 
	\le \frac{8 c_\theta }{{c_\theta}^2 + {c_\omega}^2} ,\quad
	0\le\frac{\abs*{\omega\rbra*{\lambda}}}{\abs*{\varPsi\rbra*{\lambda}}^2} \lambda^\alpha L\rbra*{\lambda} 
	\le \frac{8\abs*{ c_\omega } +4}{{c_\theta}^2 + {c_\omega}^2} .
\end{align}
Note that 
\begin{align}
	\frac{\theta\rbra*{\lambda}}{\abs*{\varPsi\rbra*{\lambda}}^2} \lambda^\alpha L\rbra*{\lambda}
	\rightarrow\frac{c_\theta}{{c_\theta}^2 + {c_\omega}^2}\quad\text{ and } \quad
	&\frac{\omega\rbra*{\lambda}}{\abs*{\varPsi\rbra*{\lambda}}^2} \lambda^\alpha L\rbra*{\lambda}
	\rightarrow\frac{c_\omega}{{c_\theta}^2 + {c_\omega}^2}\quad\text{ as }\, \lambda\rightarrow\infty.
\end{align}
The integrals for large $\lambda$ can be expressed as follows:
	\begin{align}
		&\frac{L\rbra*{\abs*{x}^{-1}}}{\abs*{x}^{\alpha-1}}\int_r^\infty\rbra*{1-\cos\lambda x}
		\frac{\theta\rbra*{\lambda}}{\abs*{\varPsi\rbra*{\lambda}}^2}\cd\lambda \notag\\
		=&\,\int_{r\abs*{x}}^\infty\frac{1-\cos \lambda}{\lambda^\alpha}\cdot\frac{\theta\rbra*{\frac{\lambda}{\abs*{x}}}}
		{\abs*{\varPsi\rbra*{\frac{\lambda}{\abs*{x}}}}^2}
		\rbra*{\frac{\lambda}{\abs*{x}}}^\alpha L\rbra*{\frac{\lambda}{\abs*{x}}} 
		\frac{L\rbra*{\frac{1}{\abs*{x}}}}{L\rbra*{\frac{\lambda}{\abs*{x}}}} \cd\lambda
	\end{align}
	and
	\begin{align}
		&\frac{L\rbra*{\abs*{x}^{-1}}}{\abs*{x}^{\alpha-1}}
		\int_r^\infty\sin\lambda x\frac{\omega\rbra*{\lambda}}{\abs*{\varPsi\rbra*{\lambda}}^2}\cd\lambda \notag\\
		=&\, \frac{L\rbra*{\abs*{x}^{-1}}}{\abs*{x}^{\alpha-1}} \text{sgn}\rbra*{x}
		\int_r^\infty\sin\lambda\abs*{x}\frac{\omega\rbra*{\lambda}}{\abs*{\varPsi\rbra*{\lambda}}^2}\cd\lambda\notag\\
		=&\, \text{sgn} \rbra*{x} \int_{r\abs*{x}}^\infty\ \frac{\sin \lambda}{\lambda^\alpha} \cdot
		\frac{\omega\rbra*{\frac{\lambda}{\abs*{x}}}}
		{\abs*{\varPsi\rbra*{\frac{\lambda}{\abs*{x}}}}^2}
		\rbra*{\frac{\lambda}{\abs*{x}}}^\alpha L\rbra*{\frac{\lambda}{\abs*{x}}} 
		\frac{L\rbra*{\frac{1}{\abs*{x}}}}{L\rbra*{\frac{\lambda}{\abs*{x}}}} \cd\lambda.
	\end{align}
If we can exchange the limit and the integral, we have 
	\begin{align}
		&\lim_{x\rightarrow0}\frac{L\rbra*{\abs*{x}^{-1}}}{\abs*{x}^{\alpha-1}}
		\int_r^\infty\rbra*{1-\cos\lambda x}\frac{\theta\rbra*{\lambda}}{\abs*{\varPsi\rbra*{\lambda}}^2}\cd\lambda 
		= \frac{c_\theta}{{c_\theta}^2+{c_\omega}^2} \int_0^\infty 
		\frac{1-\cos \lambda}{\lambda^\alpha} \cd \lambda 
		= \frac{\pi C_\alpha c_\theta}{{c_\theta}^2+{c_\omega}^2} \label{thm1.2.lim.cos}
	\end{align}
	and
	\begin{align}
		&\lim_{x\rightarrow0\pm}\frac{L\rbra*{\abs*{x}^{-1}}}{\abs*{x}^{\alpha-1}}
		\int_r^\infty\sin\lambda x\frac{\omega\rbra*{\lambda}}{\abs*{\varPsi\rbra*{\lambda}}^2}\cd\lambda 
		= \pm\frac{c_\omega}{{c_\theta}^2+{c_\omega}^2} 
		\int_0^\infty \frac{\sin\lambda}{\lambda^\alpha} \cd \lambda 
		= \pm\frac{\pi C_\alpha c_\omega}{{c_\theta}^2+{c_\omega}^2} \tan\frac{\pi(\alpha-1)}{2}. \label{thm1.2.lim.sin}
	\end{align}
Here we used Eq.~\eqref{Calpha} and its variation: 
\begin{align}
	&\frac{1}{\pi} \int_0^\infty \frac{\sin x}{x^\beta} \cd x = C_\beta \tan\frac{\pi\rbra*{\beta-1}}{2}\quad
	\rbra{\beta\in\rbra*{1,2}},
\end{align}
which is obtained from Eq.~\eqref{Calpha} by integration by parts. 
Let us justify the exchange by the dominated convergence theorem. 
Note that, for $\abs*{x}\in\rbra*{0,r^{-1}}$, we have, by $ \delta\in\rbra*{0,\rbra*{\alpha-1}\land\rbra*{2-\alpha}} $ 
	\begin{align}
		0\le&\,\frac{1-\cos \lambda}{\lambda^\alpha}\cdot\frac{\theta\rbra*{\frac{\lambda}{\abs*{x}}}}
		{\abs*{\varPsi\rbra*{\frac{\lambda}{\abs*{x}}}}^2}
		\rbra*{\frac{\lambda}{\abs*{x}}}^\alpha L\rbra*{\frac{\lambda}{\abs*{x}}} 
		\frac{L\rbra*{\frac{1}{\abs*{x}}}}{L\rbra*{\frac{\lambda}{\abs*{x}}}} \cdot1_{\cbra*{\lambda\ge r\abs*{x}}}\notag\\
		\le&\,\frac{16c_\theta}{{c_\theta}^2+{c_\omega}^2}\cdot
		\frac{1-\cos\lambda}{\lambda^\alpha}\max\cbra*{\lambda^\delta,\lambda^{-\delta}}\cdot1_{\{\lambda>0\}} \notag\\
		\le&\,\frac{16c_\theta}{{c_\theta}^2+{c_\omega}^2}\cdot
		\frac{1-\cos\lambda}{\lambda^{\alpha+\delta}}\cdot1_{\cbra*{\lambda\in\rbra*{0,1}}}+ 
		\frac{32c_\theta}{{c_\theta}^2+{c_\omega}^2}\cdot\frac{1}{\lambda^{\alpha-\delta}}
		\cdot1_{\cbra*{\lambda\ge1}}
		\in\mathcal{L}^1 
	\end{align}
	and
	\begin{align}
		0\le&\,\frac{\abs*{\sin\lambda}}{\lambda^\alpha}\cdot\frac{\abs*{\omega\rbra*{\frac{\lambda}{\abs*{x}}}}}
		{\abs*{\varPsi\rbra*{\frac{\lambda}{\abs*{x}}}}^2}
		\rbra*{\frac{\lambda}{\abs*{x}}}^\alpha L\rbra*{\frac{\lambda}{\abs*{x}}} 
		\frac{L\rbra*{\frac{1}{\abs*{x}}}}{L\rbra*{\frac{\lambda}{\abs*{x}}}} \cdot1_{\{\lambda\ge r\abs*{x}\}}\notag\\
		\le&\,\frac{16\abs*{c_\omega}+8}{{c_\theta}^2+{c_\omega}^2}\cdot
		\frac{\abs*{\sin\lambda}}{\lambda^\alpha}
		\max\cbra*{\lambda^\delta,\lambda^{-\delta}}\cdot1_{\{\lambda>0\}} \notag\\
		\le&\,\frac{16\abs*{c_\omega}+8}{{c_\theta}^2+{c_\omega}^2}\cdot
		\frac{\abs*{\sin\lambda}}{\lambda^{\alpha+\delta}}\cdot1_{\cbra*{\lambda\in\rbra*{0,1}}}+ 
		\frac{16\abs*{c_\omega}+8}{{c_\theta}^2+{c_\omega}^2}\cdot\frac{1}{\lambda^{\alpha-\delta}}
		\cdot1_{\cbra*{\lambda\ge1}}
		\in\mathcal{L}^1.
	\end{align}
Hence, we may apply the dominated convergence theorem and obtain Eq.~\eqref{thm1.2.lim.cos} and~\eqref{thm1.2.lim.sin}. 
Consequently we conclude that
\begin{align}
	\lim_{x\rightarrow0\pm}\frac{L\rbra*{\abs*{x}^{-1}}}{\abs*{x}^{\alpha-1}}h\rbra*{x} 
	= \frac{C_\alpha}{{c_\theta}^2+{c_\omega}^2} 
	\rbra*{c_\theta\pm c_\omega\cot\frac{\pi\alpha}{2}}.\label{eq.04} 
\end{align}
The proof is now complete.
\end{proof}
The following theorem shows that $h_q$ for $q>0$ has the same asymptotic behavior as $h$. 
\begin{Thm} \label{thm3.1}
	Suppose that the assumptions of Theorem~\ref{thm1.2} hold. Then, for every $q>0$, it holds that 
	\begin{align}
		\frac{h_q\rbra*{x}}{h\rbra{x}}\rightarrow 1\quad(x\rightarrow0\pm), \quad \text{ when $c_\pm\neq0$},
	\end{align}
	where $c_+$ and $c_-$ are defined by~\eqref{eq.c+-}.
\end{Thm}
We omit the proof because it is quite similar to that of Theorem~\ref{thm1.2}.

\section{Regular variation of the L\'{e}vy-Khinchin exponents} \label{sec4}
Suppose that the assumptions of Theorem~\ref{thm1.3} hold. Recall that for $ \lambda \in \mathbb{R} $
	\begin{align}
		\theta\rbra*{\lambda} =\int_\mathbb{R} \rbra*{1 - \cos \lambda x} \nu\rbra*{\cd x},\, \quad 
		\omega\rbra*{\lambda} =b\lambda + 
		\int_\mathbb{R} \rbra*{\lambda x\cdot1_{\cbra{\abs*{x}<1}} - \sin \lambda x} \nu\rbra*{\cd x}. 
	\end{align}
\begin{proof}[Proof of Theorem~\ref{thm1.3}]
Let $r > 0$. 
For the integrals for large $\abs{x}$, we have 
\begin{align}
	&\lim_{\lambda\rightarrow\infty} \frac{1}{\lambda^\alpha K_\pm\rbra*{\lambda^{-1}}} 
	\int_{\{\abs*{x}>r\}} \rbra*{1-\cos\lambda x} \nu\rbra*{\cd x} = 0, \label{eq.06}\\
	&\lim_{\lambda\rightarrow\infty} \frac{1}{\lambda^\alpha K_\pm\rbra*{\lambda^{-1}}} \rbra*{ b\lambda + 
	\int_{\{\abs*{x}>r\}} \rbra*{\lambda x\cdot1_{\cbra{\abs*{x}<1}}-\sin\lambda x} \nu\rbra*{\cd x}} = 0,\label{eq.07}
\end{align}
since
\begin{align}
	0\le\int_{\{\abs*{x}>r\}} \rbra*{1-\cos\lambda x} \nu\rbra*{\cd x} &\le 2\nu\rbra*{[-r,r]^c}
\end{align}
and
\begin{align}
	\int_{\{\abs*{x}>r\}} \abs*{\lambda x\cdot1_{\cbra{\abs*{x}<1}}-\sin\lambda x }\nu\rbra*{\cd x} 
	&\le (\lambda+1)\nu\rbra*{[-r,r]^c}.
\end{align}
For the integrals for small $\abs{x}$, we shall find a constant $r\in(0,1)$ such that the following hold:
\begin{align}
	&\lim_{\lambda\rightarrow\infty} \frac{1}{\lambda^\alpha K_\pm\rbra*{\lambda^{-1}}} \int_{\{\pm x \in (0,r]\}} \rbra*{1-\cos\lambda x} \nu\rbra*{\cd x} 
	= \int_0^\infty \frac{1-\cos x}{x^{\alpha+1}} \cd x, \label{limoftheta}\\
	&\lim_{\lambda\rightarrow\infty} \frac{1}{\lambda^\alpha K_\pm\rbra*{\lambda^{-1}}} \int_{\{\pm x \in (0,r]\}} \rbra*{\lambda x-\sin\lambda x} \nu\rbra*{\cd x} 
	= \pm \int_0^\infty \frac{x-\sin x}{x^{\alpha+1}} \cd x. \label{limofomega}
\end{align} 
To check these, note that 
\begin{align}
	&\frac{1}{\lambda^\alpha K_\pm\rbra*{\lambda^{-1}}}
	\int_{\{\pm{x}\in(0,r]\}} \rbra*{1-\cos\lambda x} \nu\rbra*{\cd x} 
	= \int_0^{r\lambda} \frac{1-\cos x}{x^{\alpha+1}} \cdot
	\frac{\xi_\pm\rbra*{\frac{x}{\lambda}}}{\rbra*{\frac{x}{\lambda}}^{-\alpha-1}K_\pm\rbra*{\frac{x}{\lambda}}}  \cdot
	 \frac{K_\pm\rbra*{\frac{x}{\lambda}}}{K_\pm\rbra*{\frac{1}{\lambda}}} \cd x ,\\
	&\frac{1}{\lambda^\alpha K_\pm\rbra*{\lambda^{-1}}}
	\int_{\{\pm{x}\in(0,r]\}} \rbra*{\lambda x-\sin\lambda x} \nu\rbra*{\cd x} 
	= \pm \int_0^{r\lambda} \frac{x - \sin x}{x^{\alpha+1}}\cdot
	\frac{\xi_\pm\rbra*{\frac{x}{\lambda}}}{\rbra*{\frac{x}{\lambda}}^{-\alpha-1}K_\pm\rbra*{\frac{x}{\lambda}}}  \cdot
	 \frac{K_\pm\rbra*{\frac{x}{\lambda}}}{K_\pm\rbra*{\frac{1}{\lambda}}} \cd x. 
\end{align}
If we can exchange the limit and the integral, then~\eqref{limoftheta} and~\eqref{limofomega} hold. 
We justify the exchange by the dominated convergence theorem. Let $r_1\in(0,1)$ be such that 
$\frac{1}{2}\le\frac{\xi_\pm\rbra*{x}}{x^{-\alpha-1}K_\pm\rbra*{x}}\le2$ for $0<x<r_1$. 
By Potter's theorem, for fixed 
$0< \delta <\rbra*{\alpha-1}\land\rbra*{2-\alpha} $, there exists a sufficiently small $r_2\in(0,1)$ such that 
\begin{align}
	\frac{K_\pm\rbra*{y^{-1}}}{K_\pm\rbra*{x^{-1}}} 
	\le 2 \max\cbra*{\rbra*{\frac{y}{x}}^\delta,\rbra*{\frac{y}{x}}^{-\delta}} \quad (x,\,y\ge {r_2}^{-1}).
\end{align}
Now take $r=\min\cbra*{r_1,r_2}$. Note that if $\lambda\ge r^{-1}\ge{r_2}^{-1}$ and $0<x\le r\lambda$, then
 $x^{-1}\lambda\ge r^{-1}\ge{r_2}^{-1}$ holds and  
$\frac{K_\pm\rbra*{\frac{x}{\lambda}}}{K_\pm\rbra*{\frac{1}{\lambda}}} \le 2 \max\cbra*{x^\delta,x^{-\delta}}$ holds. 
Hence, for $\lambda\ge r^{-1}$, it holds that
\begin{align}
	0\le&\,\frac{1-\cos x}{x^{\alpha+1}} \cdot\frac{\xi_\pm\rbra*{\frac{x}{\lambda}}}
	{\rbra*{\frac{x}{\lambda}}^{-\alpha-1}K_\pm\rbra*{\frac{x}{\lambda}}} \cdot
	\frac{K_\pm\rbra*{\frac{x}{\lambda}}}
	{K_\pm\rbra*{\frac{1}{\lambda}}} \cdot1_{\{x\in(0,r\lambda]\}}  \notag\\
	\le&\, 4 \rbra*{1-\cos x} x^{-\alpha-1}\max\cbra*{x^\delta,x^{-\delta}} \cdot1_{\{x>0\}} \notag\\
	\le&\,\frac{4\rbra*{1-\cos x}}{x^{\alpha+1+\delta}}\cdot1_{\{x\in(0,1]\}}+\frac{8}{x^{\alpha+1-\delta}}
	\cdot1_{\{x>1\}} \in \mathcal{L}^1
\end{align}
and
\begin{align}
	0\le&\,\rbra*{x-\sin x} 
	\frac{\xi_\pm\rbra*{\frac{x}{\lambda}}}{\rbra*{\frac{x}{\lambda}}^{-\alpha-1}K_\pm\rbra*{\frac{x}{\lambda}}} \cdot
	\frac{\rbra*{\frac{x}{\lambda}}^{-\alpha-1} K_\pm\rbra*{\frac{x}{\lambda}}}
	{\rbra*{\frac{1}{\lambda}}^{-\alpha-1} K_\pm\rbra*{\frac{1}{\lambda}}} \cdot1_{\{x\in(0,r\lambda]\}} \notag\\
	\le&\, 4\rbra*{x-\sin x} x^{-\alpha-1}\max\cbra*{x^\delta,x^{-\delta}} \cdot1_{\{x>0\}} \notag\\
	\le&\, \frac{4\rbra*{x-\sin x}}{x^{\alpha+1+\delta}} \cdot1_{\{x\in(0,1]\}} 
	+ \frac{8}{x^{\alpha-\delta}} \cdot1_{\{x>1\}} 
	\in \mathcal{L}^1.
\end{align}
The integrability mentioned above follows from $0< \delta <\rbra*{\alpha-1}\land\rbra*{2-\alpha} $.
Hence, it follows by the dominated convergence theorem that the convergences Eq.~\eqref{limoftheta} 
and~\eqref{limofomega} hold.
Recall that $k = \lim_{x\rightarrow0+} \frac{K_-\rbra*{x}}{K_+\rbra*{x}} \in[0,\infty]$. 
Then it can be concluded that
\begin{align}
	\frac{\theta\rbra*{\lambda}}{\lambda^\alpha\rbra*{K_+\rbra*{\lambda^{-1}}+K_-\rbra*{\lambda^{-1}}}} 
	\xrightarrow[\lambda\rightarrow\infty]{} \int_0^\infty \frac{1-\cos x}{x^{\alpha+1}} \cd x  
	= \pi C_{\alpha+1} \label{eq.08}
\end{align}
and
\begin{align} 
	\frac{\omega\rbra*{\lambda}}{\lambda^\alpha\rbra*{K_+\rbra*{\lambda^{-1}}+K_-\rbra*{\lambda^{-1}}}} 
	\xrightarrow[\lambda\rightarrow\infty]{} 
	\frac{1-k}{1+k}\int_0^\infty \frac{x-\sin x}{x^{\alpha+1}} \cd x 
	=\frac{1-k}{1+k}\cdot\frac{\pi C_\alpha}{\alpha}. \label{eq.09} 
\end{align}
Here we used the formula 
\begin{align}
	&\frac{1}{\pi} \int_0^\infty \frac{x-\sin x}{x^{\alpha+1}} \cd x = \frac{C_\alpha}{\alpha}\quad
	\rbra{\alpha\in\rbra*{1,3}}, \label{x-sin.C_}
\end{align}
which is obtained from Eq.~\eqref{Calpha} by integration by parts.
The proof is now complete.
\end{proof}

\section{Asymptotic behavior of the renormalized zero resolvents with positive Gaussian coefficient} \label{sec5}
In this section, we consider the case where the Gaussian coefficient $a>0$. 
Recall that $\theta$ and $\omega$ are given as
\begin{align}
\theta\rbra*{\lambda} &= \Re\varPsi\rbra*{\lambda} = a\lambda^2 + \int_\mathbb{R} \rbra*{1 - \cos \lambda x} \nu\rbra*{\cd x} ,  \\
\omega\rbra*{\lambda} &= \Im\varPsi\rbra*{\lambda} = b\lambda + \int_\mathbb{R} \rbra*{\lambda x \cdot1_{\{\abs*{x} < 1\}} - \sin \lambda x} \nu\rbra*{\cd x}.  
\end{align}
Note that 
\begin{align}
	\lim_{\lambda\rightarrow\infty} \frac{\theta(\lambda)}{\lambda^2} = a , \quad
	\lim_{\lambda\rightarrow\infty} \frac{\omega(\lambda)}{\lambda^2} = 0. \label{lim/2}
\end{align}
In fact, we may apply the dominated convergence theorem because
\begin{align}
	&\sup_{\lambda>1}\rbra*{\frac{1-\cos\lambda x}{\lambda^2}}\le 
	x^2\cdot1_{\cbra*{\abs{x}<1}}+2\cdot1_{\cbra*{\abs{x}\ge1}}\in\mathcal{L}^1(\nu),\\
	&\sup_{\lambda>1}\rbra*{\frac{\lambda x \cdot1_{\{\abs*{x} < 1\}} - \sin \lambda x}{\lambda^2}} 
	\le x^2 \cdot1_{\cbra*{\abs{x}<1}} + 1_{\cbra*{\abs{x}\ge1}} \in\mathcal{L}^1(\nu).
\end{align}
Here we used the inequality $\abs*{x-\sin x}\le x^2$ for $x\in\mathbb{R}$.
Note also that Assumption~\ref{ass.A} is automatically satisfied, since
\begin{align} 
	\abs*{q+\varPsi(\lambda)}\ge q \quad(\lambda\ge0)\quad\text{and}\quad
	\frac{\abs*{q+\varPsi(\lambda)}}{\lambda^2}\rightarrow a\quad(\lambda\rightarrow\infty). \label{check.for.A.a>0}
\end{align} 
Let us consider the following assumption: 
\begin{align}
\text{The Gaussian coefficient $a>0$  and} \int_0^\infty \abs*{\Im\rbra*{\frac{\lambda}{\varPsi(\lambda)}}} \cd \lambda <\infty. \label{ass.Z}
\end{align} 
This assumption is stronger than Assumption~\ref{ass.T}. A sufficient condition for Assumption~\eqref{ass.Z} is that $a>0$ as well as 
\begin{align}
	\int_0^1\abs*{\Im\rbra*{\frac{\lambda}{\varPsi(\lambda)}}}<\infty \quad\text{and}\quad 
	\lim_{\lambda\rightarrow\infty} \frac{\omega(\lambda)}{\lambda^{2-\epsilon}} = 0 \quad\text{for some $\epsilon>0$.}
\end{align}
In fact, the following holds: 
	\begin{align}
		\abs*{\Im\rbra*{\frac{\lambda}{\varPsi(\lambda)}}}\cdot \lambda^{1+\epsilon}
		= \frac{\abs*{\omega(\lambda)}}
		{\lambda^{2-\epsilon}}\cdot\frac{\lambda^4}{\abs*{\varPsi(\lambda)}^2} 
		\xrightarrow[\lambda\rightarrow\infty]{} 0.
	\end{align}
\begin{Thm}\label{thm5.1}
	Suppose that Assumption~\eqref{ass.Z} holds. Then 	it holds that
	\begin{align}
		h(x) \sim c_\pm \abs{x} \quad(x\rightarrow0\pm),
	\end{align}
	where the coefficients $c_+$ and $c_-$ are given as
	\begin{align}
		c_\pm = \frac{1}{2a}\pm\frac{1}{\pi}\int_0^\infty \Im\rbra*{\frac{\lambda}{\varPsi(\lambda)}}\cd\lambda. \label{thm5.1.c+-}
	\end{align}
\end{Thm}
\begin{proof}
	Let $a>0$. Note that 
	\begin{align}
		\lim_{\lambda\rightarrow\infty} \Re\rbra*{\frac{\lambda^2}{\varPsi(\lambda)}} = a^{-1} 
	\end{align}
	holds. Hence, there exists $r>0$ such that 
	\begin{align}
		\frac{a^{-1}}{2} \le \Re\rbra*{\frac{\lambda^2}{\varPsi(\lambda)}} \le 2 a^{-1} \quad(\lambda\ge r).
	\end{align}
	If we can exchange the limit and the integral, then we have 
	\begin{align}
		\lim_{x\rightarrow0}\frac{1}{\abs*{x}}
		\int_r^\infty\rbra*{1-\cos\lambda x}\Re\rbra*{\frac{1}{\varPsi(\lambda)}} \cd \lambda
		= \lim_{x\rightarrow0} \int_{r\abs{x}}^\infty\frac{1-\cos\lambda}{\lambda^2}
		\Re\rbra*{\frac{\rbra*{\frac{\lambda}{\abs{x}}}^2}{\varPsi\rbra*{\frac{\lambda}{\abs{x}}}}} \cd \lambda
		= \frac{\pi}{2a} \label{limr}
	\end{align}
	and
	\begin{align}
		\lim_{x\rightarrow0\pm}\frac{1}{\abs*{x}}
		\int_0^\infty\sin\lambda x\Im\rbra*{\frac{1}{\varPsi(\lambda)}} \cd \lambda
		=\pm \int_0^\infty\Im\rbra*{\frac{\lambda}{\varPsi(\lambda)}} \cd\lambda.\label{limi}
	\end{align}
	To justify the exchange, note that the following hold:
	\begin{align}
		0\le\frac{1-\cos\lambda}{\lambda^2}
		\Re\rbra*{\frac{\rbra*{\frac{\lambda}{\abs{x}}}^2}{\varPsi\rbra*{\frac{\lambda}{\abs{x}}}}} 
		\cdot1_{\cbra*{\lambda\ge r\abs{x}}}
		\le 2a^{-1}\frac{1-\cos\lambda}{\lambda^2} \cdot1_{\cbra{\lambda>0}} 
		\in\mathcal{L}^1
	\end{align}
	and
	\begin{align}
		\abs*{\frac{\sin\lambda x}{\abs*{x}} \Im\rbra*{\frac{1}{\varPsi(\lambda)}}} 
		\le \abs*{\Im\rbra*{\frac{\lambda}{\varPsi(\lambda)}}}
		\in \mathcal{L}^1.
	\end{align}	
	Hence, by the dominated convergence theorem, we see that Eq.~\eqref{limr} and~\eqref{limi} hold.
	By Lemma~\ref{lem2.2}, we have
	\begin{align}
		\frac{1}{\abs*{x}}
		\int_0^r\abs*{\rbra*{1-\cos\lambda x}\Re\rbra*{\frac{1}{\varPsi(\lambda)}}} \cd \lambda
		\le \abs{x} \int_0^r \frac{\lambda^2}{\abs{\varPsi(\lambda)}}\cd \lambda
		\xrightarrow[x\downarrow0]{} 0.
	\end{align}
	Consequently, we conclude that
	\begin{align}
		\lim_{x\rightarrow0\pm} \frac{1}{\abs*{x}}h(x) 
		= \frac{1}{2a}\pm\frac{1}{\pi}\int_0^\infty \Im\rbra*{\frac{\lambda}{\varPsi(\lambda)}}\cd\lambda.
	\end{align}
\end{proof}
The same result as Theorem~\ref{thm3.1} holds under Assumption~\eqref{ass.Z}, 
which can be proved by mimicking the proof of Theorem~\ref{thm5.1}.

\section{Appendix: The asymptotic behavior of the renormalized zero resolvent at $\infty$} \label{sec6}
We consider the asymptotic behavior of the renormalized zero resolvent at $\infty$ under the regular varying condition on 
the exponents $\theta$ and $\omega$ near $0$. As the proofs are quite similar to those of Theorems~\ref{thm1.2} and~\ref{thm1.3}, we state several results without proofs. 
\subsection{The asymptotic behavior of the exponents $\theta$ and $\omega$ near $0$}
In this section, we will consider the following assumption:
\begin{enumerate} [label = \textbf{(B)}]
\item \label{ass.v}
	Suppose that $\nu(\cd x) = \xi(x)\cd x$ with $\xi_\pm(x)\coloneqq\xi(\pm x)$ for $x\ge0$ and that
	\begin{align}
		\xi_\pm(x) \sim x^{-\beta-1}K_\pm(x)\quad(x\rightarrow\infty),\notag
	\end{align} 
	where $K_\pm$ are slowly varying functions at $\infty$ and $\beta$ is a constant. 
\end{enumerate}
Note that if~\ref{ass.v} holds, then $\beta\ge 0$ holds by the assumption $\int_\mathbb{R}\rbra{x^2\land1}\,\nu(\cd x)<\infty$.
\begin{Thm}\label{thm6.1}
	The following assertions hold.
	\begin{enumerate}
		\item \label{thm6.1.i}
			If Assumption~\ref{ass.v} holds with $\beta\in(0,2)$, 
			then we have
			\begin{align}
				\theta(\lambda) \sim c_\theta^0 \lambda^\beta L(\lambda) \quad (\lambda\rightarrow 0+),
			\end{align}
			where $c_\theta^0 = \pi C_{\beta+1}$ and the slowly varying function $L$ at $0$ 
			is given as 
			\begin{align}
				L(\lambda)\coloneqq K_+\rbra*{\lambda^{-1}}+K_-\rbra*{\lambda^{-1}}. \label{eq.6.1.L}
			\end{align} 
			Consequently, Eq.~\eqref{int.ass.T} holds.
		\item \label{thm6.1.ii}
			If $\int_{\mathbb{R}} x^2\,\nu(\cd x)<\infty$, then we have
			\begin{align}
				\theta(\lambda) \sim c_\theta^2 \lambda^2 \quad (\lambda\rightarrow 0+), 
				\label{thm6.1.sim.ii}
			\end{align}
			where $c_\theta^2 = a + \frac{1}{2}\int_{\mathbb{R}} x^2\, \nu(\cd x)$. 
		\item \label{thm6.1.iii}
			If Assumption~\ref{ass.v} holds with $\beta\in(0,1)$ 
			and if $k_0\coloneqq \lim_{x\rightarrow\infty} \frac{K_-(x)}{K_+(x)}\in[0,\infty]$ holds, then we have
			\begin{align}
				\omega(\lambda) \sim c_\omega^0 \lambda^\beta L(\lambda) \quad (\lambda\rightarrow 0+),
			\end{align}
			where $c_\omega^0 = -\frac{1-k_0}{1+k_0}\pi C_{\beta+1}\tan\frac{\pi\beta}{2}$ and 
			the slowly varying function $L$ at $0$ is given as Eq.~\eqref{eq.6.1.L}. 
		\item \label{thm6.1.iv}
			If $\int_{(-1,1)^c}\abs{x}\,\nu(\cd x)<\infty$, then we have
			\begin{align}
				\omega(\lambda)\sim c_\omega^1 \lambda \quad (\lambda\rightarrow 0+), \label{thm6.1.sim.i}
			\end{align}
			where $c_\omega^1=b-\int_{(-1,1)^c} x\,\nu(\cd x)$. 
			In the degenerate case $b=\int_{(-1,1)^c} x\,\nu(\cd x)$ (i.e., $c_\omega^1=0$), 
			if Assumption~\ref{ass.v} holds for some $1<\beta<3$, and if $k_0\coloneqq \lim_{x\rightarrow\infty} \frac{K_-(x)}{K_+(x)}\in[0,\infty]$ 
			holds, then it holds that
			\begin{align}
				\omega(\lambda)\sim c_\omega^3 \lambda^\beta L(\lambda) \quad (\lambda\rightarrow 0+), 
			\end{align}
			where $c_\omega^3= \frac{1-k_0}{1+k_0}\cdot\frac{\pi C_\beta}{\beta}$ and 
			the slowly varying function $L$ at $0$ is given as Eq.~\eqref{eq.6.1.L}. 
	\end{enumerate}
\end{Thm}

\subsection{ The asymptotic behavior of the renormalized zero resolvent at $\infty$}

\begin{Thm}\label{thm6.3}
	Suppose that Assumption~\ref{ass.A} holds. Suppose, in addition, that there exist a slowly varying function $L$ at $0$ and  
	constants $c_\theta^0>0$, $c_\omega^0\in\mathbb{R}$ and $\beta\in(1,2)$ such that 
	\begin{align}
		\theta(\lambda)\sim c_\theta^0 \lambda^\beta L(\lambda) \quad \text{and} \quad 
		\omega(\lambda)\sim c_\omega^0 \lambda^\beta L(\lambda) \quad(\lambda\rightarrow0+).\label{thm6.3.sim.ass}
	\end{align}
	Then Assumption~\ref{ass.T} holds and 
	\begin{align}
		h(x) \sim c_\pm^0 \abs{x}^{\beta-1} L\rbra*{\frac{1}{\abs{x}}}^{-1}\quad(x\rightarrow \pm\infty),
		\label{thm6.1.sim.h}
	\end{align}
	where the coefficients $c_+^0$ and $c_-^0$ are given as 
	\begin{align}
		c_\pm^0 = \frac{C_\beta}{{c_\theta^0}^2+{c_\omega^0}^2} \rbra*{c_\theta^0 \pm c_\omega^0 \cot\frac{\pi\beta}{2}}. \label{c_+-^0}
	\end{align}
\end{Thm}
By~\ref{thm6.1.i} and~\ref{thm6.1.iv} of Theorem~\ref{thm6.1}, the following is a sufficient condition for the regular 
variation conditions~\eqref{thm6.3.sim.ass} of the exponents $\theta$ and $\omega$: 
Assumption~\ref{ass.v} holds with $\beta\in(1,2)$ as well as 
\begin{align}
	\omega(\lambda)=\int_\mathbb{R}\rbra*{\lambda x - \sin \lambda x}\,\nu(\cd x) \quad (\lambda \in\mathbb{R})\quad
	\text{and}\quad\lim_{x\rightarrow\infty} \frac{K_-(x)}{K_+(x)}\in[0,\infty].
\end{align}

\bibliographystyle{plain}

\end{document}